\documentclass[12pt,a4paper,reqno]{amsart}
\usepackage{amsmath,amsthm,amsfonts,amssymb,mathrsfs}

\usepackage{color}

\newtheorem{theorem}{Theorem}
\newtheorem{lemma}[theorem]{Lemma}

\newtheorem{prop}[theorem]{Proposition}
\newtheorem{cor}[theorem]{Corollary}

\newtheorem{rem}{Remark}

\numberwithin{equation}{section}
\numberwithin{theorem}{section}

\def \RR {\mathbb{R}}
\def \NN {\mathbb{N}}

\def\({\left(}
\def\){\right)}
\def\ls{\lesssim}
\begin{document}
\title[Solution of critical fractal Burgers equation]{Stable estimates for source solution of critical fractal Burgers equation}

\author{Tomasz Jakubowski}
\address[Tomasz Jakubowski]{Wroclaw University of Technology, ul. Wybrze\.ze Wyspia\'nskiego 27, Wroclaw, Poland}
\email{tomasz.jakubowski@pwr.edu.pl}

\author{Grzegorz Serafin}
\address[Grzegorz Serafin]{Wroclaw University of Technology, ul. Wybrze\.ze Wyspia\'nskiego 27, Wroclaw, Poland}
\email{grzegorz.serafin@pwr.edu.pl}
\begin{abstract}
In this paper, we provide two-sided estimates for  the source solution of $d$-dimensional critical fractal Burgers equation $u_t-\Delta^{\alpha/2}+b\cdot \nabla\(u|u|^q\)=0$, $q=(\alpha-1)/d$, $\alpha\in(1,2)$, $b\in\RR^d$, by the density function of the  isotropic $\alpha$-stable process.
\end{abstract}

\maketitle

\section{Introduction}

Let $d \in \NN$ and $\alpha \in (1,2)$. We consider the following pseudo-differential equation
\begin{equation}\label{eq:problem}
\left\{\begin{array}{l}u_t-\Delta^{\alpha/2}u+b\cdot \nabla\(u|u|^q\)=0, \qquad t>0, \; x \in \RR^d,\\
u(0,x)=M\delta_0(x),\end{array}\right.
\end{equation}
where $M>0$ is arbitrary constant and $b \in \RR^d$ is a constant vector. In this paper, we focus on the critical case $q=(\alpha-1)/d$. Here, $\Delta^{\alpha/2}$ denotes the fractional Laplacian defined by the Fourier transform
$$
\widehat{\Delta^{\alpha/2} \phi}(\xi)=-|\xi|^{\alpha}\widehat\phi(\xi),\ \ \ \ \phi\in C_c^\infty(\RR^d).
$$
Equation (\ref{eq:problem}) for various values of $q$ and initial conditions $u_0$ was recently intensely studied (\cite{BFW, BKW2001, BKW2001Studia, BK2006}). For $d=1$, the case $q=2$ is of particular interest (see e.g. \cite{KMX, AIK, KNS, WW}) because it is a natural counterpart of the classical Burgers equation. Another interesting value of $q$ is $\frac{\alpha-1}{d}$. 
In \cite{BKW2001} authors proved that the solution of  (\ref{eq:problem}), which we  denote throughout the paper  by $u_M(t,x)$, exists and is unique and positive. It belongs also to $L^p\(\RR^d\)$ for every $p\in[1,\infty]$. The exponent $q=\frac{\alpha-1}{d}$ is critical in some sense. It is the only value for which the function $u_M(t,x)$ is self-similar. Is satisfies the following scaling condition (\cite{BKW2001})
\begin{equation}
\label{eq:scaling}
u_M(t,x)=a^{d}u_M(a^{\alpha}t,ax), \hspace{15mm}\text{for all }a>0.
\end{equation}
Furthermore, the linear and the nonlinear terms in (\ref{eq:problem}) have equivalent influence on the asymptotic behavior  of the solution. If  $q>(\alpha-1)/d$, the operator $\Delta^{\alpha/2}$ plays the main role. More precisely, for such $q$ and  a function  $u$ satisfying   (\ref{eq:problem}),  with not necessarily the same initial condition, we have 
$$\lim_{t\rightarrow\infty}t^{n(1-1/p)/\alpha}\left\|u(t,\cdot)-e^{\Delta^{\alpha/2}}u(0,\cdot)\right\|_p=0,\ \ \ \ \text {for each } p\in[1,\infty].$$   
For $q<(\alpha-1)/d$ another asymptotic behavior is expected. In addition, taking $q=\frac{\alpha-1}{d}$ for $d=1$ and $\alpha=2$ we obtain the classical case, which makes the equation (\ref{eq:problem}) with critical exponent $q$ one of the natural generalizations of the Burgers equation.

Till the end of the paper we assume that $d \ge 1$, $\alpha \in (1,2)$ and $q=\frac{\alpha-1}{d}$. Let $p(t,x)$ be the fundamental solution of
\begin{equation}\label{eq:fundamental}
v_t=\Delta^{\alpha/2} v.
\end{equation}
In \cite{BK2006} the authors proved that for sufficiently small $M$ there is a constant $C=C(d,\alpha,M,b)$ such that
\begin{equation}\label{eq:mainresult}
u_M(t,x)\le Cp(t,x), \qquad t>0,\; x\in \RR^d.
\end{equation}
In this paper we get rid of the smallness assumption of $M$. Furthermore, we also obtain the lower bounds of $u_M$. We propose a new method which allows us to show pointwise estimates of solutions to the nonlinear problem (\ref{eq:problem}) without the smallness assumption imposed on $M$. This method has been inspired by the proof of \cite[Theorem 1]{BJ2}. Our main result is

\begin{theorem}\label{mainthm}
Let $d \ge 1$ and $\alpha \in (1,2)$. Let $u_M(t,x)$ be the solution of the equation (\ref{eq:problem}) with $q = \frac{\alpha-1}{d}$. There exists a constant $C=C(d,\alpha,M,b)$ such that
$$
C^{-1} p(t,x) \le u_M(t,x) \le C p(t,x), \qquad t>0,\; x \in \RR^d.
$$
\end{theorem}

The fractional Laplacian plays also a very important role in the probability theory as a generator of the so called isotropic stable process. The theory of its linear perturbations has been recently significantly developed, see e.g., \cite{BJ1, BJ2, 2009-TJ-KS-jee, JS2, Sz, MM1, MM2, CKS2012, GR}. However,  since the term $b\cdot\nabla (|u|^q u)$ in (\ref{eq:problem}) represents a nonlinear drift, methods used in the linear case often cannot be  adapted. In the proofs we mostly use the Duhamel formula and its suitable iteration. The scaling condition (\ref{eq:scaling}) is also intensively exploit.

The paper is organized as follows.  In Preliminaries we collect some basic properties of the function $p(t,x)$ and introduce the Duhamel formula as well. In Section 3 we prove that the solution of (\ref{eq:problem}) converges to $0$ as $|x| \to \infty$. In section 4 we prove Theorem \ref{mainthm}.

\section{Preliminaries}

\subsection{Notation}
For two positive functions $f,g$ we denote $f\ls g$ whenever there exists a constant $c>1$ such that $f(x)<cg(x)$ for every argument $x$.  If $f\ls g$ and $g\ls f$ we write $f\approx g$. If value of a constant in estimates is relevant, we denote it by $C_k$, $k\in\mathbb N$, and it does not change throughout the paper.

\subsection{Properties of $p(t,x)$}
 The fundamental solution of (\ref{eq:fundamental}) may be given by the inverse Fourier transform 
\begin{equation*}
	p(t,x) = (2\pi)^{-d} \int_{\RR^d} e^{-i x \cdot \xi} e^{-t |\xi|^{\alpha}}d\xi, \qquad t>0\,,\; x \in \RR^d\,.
\end{equation*}
This implies the following scaling property
\begin{equation}
\label{eq:scalingp}
p(t,x)=a^{d}p(a^{\alpha}t,ax), \hspace{15mm}\text{for all }a>0.
\end{equation}
Note that $u_M(t,x)$ possesses exactly the same property. Let $p(t,x,y) := p(t,y-x)$. Below, we give two well-known estimates of $p$ and the gradient of $p$ (see \cite{BJ1} for more details).
\begin{align}\label{est:p}
p(t,x,y)&\approx \frac t{\(t^{1/\alpha}+|y-x|\)^{d+\alpha}},\\
\label{est:gradp}
|\nabla_yp(t,x,y)|&\approx \frac {t\, |y-x|}{\(t^{1/\alpha}+|y-x|\)^{d+2+\alpha}}.
\end{align}
We will need the following lemma
\begin{lemma}
\label{est:inteps} For $t, \varepsilon >0$ we have
\begin{equation*}
\int_{B(0,\varepsilon)}p(t,0,w)dw\approx\(\frac{\varepsilon}{t^{1/\alpha}+\varepsilon}\)^{d}.
\end{equation*}
\end{lemma}
\begin{proof} By formula (\ref{est:p}),
\begin{align}
\int_{B(0,\varepsilon)}\nonumber
p(t,0,w)dw&\approx
\int_{B(0,\varepsilon)} \frac {t\,dw}{\(t^{1/\alpha}+|w|\)^{d+\alpha}}\\\nonumber
&=c\int_0^\varepsilon \frac {t\,r^{d-1}dr}{\(t^{1/\alpha}+r\)^{d+\alpha}}\\\label{eq:int}
&=c\,\frac{t}{\varepsilon^\alpha}\int_0^1 \frac {r^{d-1}dr}{\(\frac{t^{1/\alpha}}{\varepsilon}+r\)^{d+\alpha}}.
\end{align}
If $t^{1/\alpha}\ge\varepsilon$, we estimate the denominator in the integral by $\frac{t^{1/\alpha}}{\varepsilon}$ and we get 
\begin{equation}\label{eq:t>e}
\int_{B(0,\varepsilon)}p(t,0,w)dw\approx \(\frac{\varepsilon}{t^{1/\alpha}}\)^d,\ \ \ \ t^{1/\alpha}\ge\varepsilon.
\end{equation}
In the case $t^{1/\alpha}< \varepsilon$ we substitute $r=\frac{t^{1/\alpha}}{\varepsilon} u$ in   (\ref{eq:int}), then
\begin{equation}\label{eq:t<e}
\int_{B(0,\varepsilon)}
p(t,0,w)dw\approx\int_0^{\varepsilon/t^{1/\alpha}} \frac {u^{d-1}du}{\(1+u\)^{d+\alpha}}\approx1,\ \ \ \ t^{1/\alpha}< \varepsilon.
\end{equation}
Combining (\ref{eq:t>e}) and (\ref{eq:t<e}), we obtain the assertion of the lemma. 
\end{proof}

\subsection{Duhamel formula} Our mail tool is the following Duhamel formula,
\begin{equation}
u_M(t,x) = M p(t,x) + \int_0^t \int_{\RR^d} p(t-s,x,z) b \cdot \nabla_z [u_M(s,z)]^{q+1}\,dz\,ds.
\label{eq:Duhamel1}
\end{equation}
In the following, we assume that $u_M(t,x)=t^{-d/\alpha}u_M(1,xt^{-1/\alpha})$ is a nonnegative self-similar solution of the equation (\ref{eq:Duhamel1}) such that $u_M(1,\cdot)\in L^p\left(\RR^d\right)$ for each $p\in[1,\infty]$. As it was mentioned in Introduction, the existence of such a function was shown in \cite{BKW2001}.  

It turns out that the integral in (\ref{eq:Duhamel1}) is not absolutely convergent, but integrating by parts we obtain a more convenient form
\begin{equation}
u_M(t,x) = M p(t,x) - \int_0^t \int_{\RR^d} b \cdot \nabla_z p(t-s,x,z)  [u_M(s,z)]^{q+1}\,dz\,ds,
\label{eq:Duhamel2}
\end{equation}
which is absolutely convergent. Indeed, we have $||u_M(s,\cdot)||_p \ls s^{-d(1-1/p)/\alpha}$, $p \in [1,\infty]$ (see \cite{BKW2001}). Hence, by (\ref{est:p}) and (\ref{est:gradp}), 
\begin{align*}
&\int_0^t \int_{\RR^d}\left| b \cdot \nabla_z p(t-s,x,z)  [u_M(s,z)]^{q+1}\right| \,dz\,ds \\
& \ls \int_0^{t/2} \int_{\RR^d} | \nabla_z p(t-s,x,z)|  [u_M(s,z)]^{q+1} \,dz\,ds + \int_{t/2}^t \int_{\RR^d} |\nabla_z p(t-s,x,z)|  [u_M(s,z)]^{q+1} \,dz\,ds \\
& \ls \int_0^{t/2} \int_{\RR^d} (t-s)^{-(d+1)/\alpha}  [u_M(s,z)]^{q+1} \,dz\,ds +  \int_{t/2}^t \int_{\RR^d} (t-s)^{-1/\alpha} p(t-s,x,z)  s^{-d(q+1)/\alpha} \,dz\,ds \\
& \ls t^{-(d+1)/\alpha} \int_0^{t/2} s^{-(\alpha-1)/\alpha} ds + t^{-d(q+1)/\alpha}  \int_{t/2}^t  (t-s)^{-1/\alpha} \,ds \ls t^{-d/\alpha}.
\end{align*}
Now, using the scaling property (\ref{eq:scaling}), we obtain
\begin{align*}
& \int_0^t \int_{\RR^d} b \cdot \nabla_z p(t-s,x,z)  [u_M(s,z)]^{q+1}\,dz\,ds   \\
& = \int_0^t \int_{\RR^d} \nabla_z p(t-s,x,z)  s^{-(d-1+\alpha)/\alpha}  [u_M(1,s^{-1/\alpha} z)]^{q+1}\,dz\,ds  \\
& = \int_0^t \int_{\RR^d} s^{(1-\alpha)/\alpha} \nabla_w p(t-s,x,s^{1/\alpha} w)    [u_M(1,w)]^{q+1}\,dw\,ds  \\
& =\alpha \int_0^{t^{1/\alpha}} \int_{\RR^d} \nabla_wp(t-r^\alpha,x,r w)[ u_M(1,w)]^{q+1}\,dw\,dr. 
\end{align*}
Finally, we get
\begin{equation}
u_M(t,x) = M p(t,x) - \alpha \int_0^{t^{1/\alpha}} \int_{\RR^d} \nabla_wp(t-r^\alpha,x,r w)[ u_M(1,w)]^{q+1}\,dw\,dr
\label{eq:Duhamelscale}
\end{equation}

\section{Behaviour of $u_M(1,x)$ at infinity}
Due to the  scaling property (\ref{eq:scaling}) it suffices to consider  $u_M(t,x)$ only for $t=1$.
\begin{lemma}\label{lem:DuhamelEst}
There is a constant $C_1>0$ such that for every $x \in \RR^d$,
\begin{equation}
u_M(1,x) \le M p(1,x) + C_1 \int_0^1 \int_{\RR^d} (1-r^\alpha)^{-1/\alpha} p(1-r^\alpha,x,r w)  [u_M(1,w)]^{q+1}\,dw\,dr.
\label{eq:DuhamelEst}
\end{equation}
\end{lemma}
\begin{proof}Formulae (\ref{est:p}) and (\ref{est:gradp}) imply that 
$$|\nabla_zp(1-r,x,z)|\ls(1-r)^{-1/\alpha}p(1-r,x,z),\ \ \ \ \ \ r\in[0,1),\ \  x,z\in\RR^d,$$
and the assertion follows by (\ref{eq:Duhamelscale}).
\end{proof}

Now, we will show that the function $u_M(1,x)$ vanishes at infinity.
\begin{prop}\label{lem:limit u}
We have $\lim\limits_{|x|\rightarrow\infty}u_M(1,x)=0$.
\end{prop}
\begin{proof}
Let us rewrite  (\ref{eq:DuhamelEst}) into the form
\begin{align*}
 u_M(1,x) \ls &\   p(1,x) +  I_1(x)+I_2(x),
\end{align*}
where 
\begin{align*}
 I_1 & =\int_0^{1/2} \int_{\RR^d} (1-r^\alpha)^{-1/\alpha} p(1-r^\alpha,x,r w)  [u_M(1,w)]^{1+q}\,dw\,dr,\\
I_2 & =  \int_{1/2}^1 \int_{\RR^d} (1-r^\alpha)^{-1/\alpha} p(1-r^\alpha,x,r w)  [u_M(1,w)]^{1+q}\,dw\,dr.
\end{align*}
Note that $a_R := \int_{B(0,R)^c} [u_M(1,w)]^{1+q}dw \to 0$ as $R \to \infty$. Using (\ref{est:p}) and estimating $1-r^\alpha\approx1$, $r\in(0,1/2)$, we obtain for $|x|>R$,
\begin{align*}
 I_1 & =\int_0^{1/2}  \int_{B(0,R)}\ldots dw\,dr+\int_0^{1/2}  \int_{B(0,R)^c}\dots dw\,dr\\
&\ls\int_0^{1/2}  \int_{B(0,R)}\frac{1}{|x|^{d+\alpha}}  u_M(1,w)^{1+q}\,dw\, dr + \int_0^{1/2}  \int_{B(0,R)^c} [u_M(1,w)]^{1+q}dw\,dr\\
&\le \frac{|B(0,R)|}{|R|^{d+\alpha}}  \left\|u_M(1,\cdot)\right\|_\infty^{1+q}+a_R,
\end{align*}
which is arbitrary small for sufficiently large $R$. Consider now the integral $I_2$ and fix $\varepsilon>0$. There exist $R_1,R_2>0$ such that 
\begin{align*}
\nonumber\int_{B(0,R_1)^c}p(1,0,w)dw&<\varepsilon,\\
\int_{B(0,R_2)^c}u_M(1,w)dw&<\varepsilon^{d+1}.
\end{align*}
The latter  inequality implies that the measure of the set $\{w\in B(0,R_2)^c:u_M(1,w)>\varepsilon\}$ is less or equal to  $\varepsilon^d$. It gives us for $r\in(1/2,1)$ and $|x|>R_1+R_2$
\begin{align*}
\int_{\RR^d} &p(1-r^\alpha,x,r w) [u_M(1,w)]^{1+q}\,dw\\
&=  \frac1{r^{d}}\int_{\RR^d} p(1-r^\alpha,x,w) [u_M(1,w/r)]^{1+q}\,dw\\
&\le \frac{\left\|u_M(1,\cdot)\right\|_{\infty}^{1+q}}{r^d}\int_{B(x,R_1)^c} p(1-r^\alpha,x,w) \,dw+\frac{\varepsilon^{1+q}}{r^d}\int_{\substack{B(x,R_1),\\u_M(1,w/r)\le\varepsilon}} p(1-r^\alpha,x,w) \,dw\\
&\ \ \ +\frac1{r^d}\left\|u_M(1,\cdot)\right\|_{\infty}^{1+q}\int_{\substack{B(x,R_1),\\u_M(1,w/r)>\varepsilon}} p(1-r^\alpha,x,w) \,dw.
\end{align*}
Using scaling property  (\ref{eq:scalingp}) and substituting $w=x+(1-r^\alpha)^{1/\alpha}z$, we get 
\begin{align*}
\int_{B(x,R_1)^c} p(1-r^\alpha,x,w) \,dw=\int_{B\(0,R_1(1-r^\alpha)^{-1/\alpha}\)^c} p(1,0,z) \,dz\le \varepsilon.
\end{align*}
Moreover,
\begin{align*}
&\left|\{w\in B(x,R_1):u_M(1,w/r)>\varepsilon\}\right|=\\
&\hspace{20mm}=r^d\left|\{w\in B(x/r,R_1/r):u_M(1,w)>\varepsilon\}\right|\\
&\hspace{20mm}\le \left|\{w\in B(0,R_2)^c:u_M(1,w)>\varepsilon\}\right|\\
&\hspace{20mm}\le \varepsilon^d\le \left|B(0,\varepsilon)\right|.
\end{align*}
Hence, by the fact that $p(s,x,w)$ is a decreasing function of $|x-w|$ and  by Lemma \ref{est:inteps}, we obtain
\begin{align*}
I_2(x)&\ls\ \varepsilon+\int_{1/2}^1\int_{B(x,\varepsilon)} p(1-r^\alpha,x,w) \,dw\,dr\\
&\ls\varepsilon+\int_{1/2}^1(1-r^\alpha)^{-1/\alpha}\(\frac{\varepsilon}{(1-r^\alpha)^{1/\alpha}+\varepsilon}\)^ddr\\
&\ls\varepsilon+\int_{0}^{1-(1/2)^\alpha}u^{-1/\alpha}\(\frac{\varepsilon}{u^{1/\alpha}+\varepsilon}\)^ddu\\
&=\varepsilon+\int_{0}^{\varepsilon^{\alpha/2}}...du+\int_{\varepsilon^{\alpha/2}}^{1-(1/2)^\alpha}...du\\
&\le\varepsilon+\int_{0}^{\varepsilon^{\alpha/2}}u^{-1/\alpha}du+\frac{\varepsilon^{d/2}}{(1+\sqrt \varepsilon)^d}\int_{\varepsilon^{\alpha/2}}^{1-(1/2)^\alpha}u^{-1/\alpha}du\\
&\ls\varepsilon+\varepsilon^{(\alpha-1)/2}+\varepsilon^{d/2}.
\end{align*}
Therefore, for sufficiently large $|x|$, integrals $I_1$ and $I_2$ are arbitrary small. This ends the proof.
\end{proof}

\section{Proof of Theorem \ref{mainthm}}
In this section, we prove the main theorem of the paper. First, we define some auxiliary functions. Let
\begin{align}\label{def:H}
H(x,w) =& \int_0^1 (1-r^\alpha)^{-1/\alpha} p(1-r^\alpha,x,r w) dr.
\end{align}
For $\beta\in (0,1)$, we define 
\begin{align}
\label{def:tildeH}
\tilde H(x,w) =& \int_0^1 r^{-\beta}(1-r^\alpha)^{-1/\alpha} p(1-r^\alpha,x,r w) dr.
\end{align}
Additionally, for $R>0$, we denote
\begin{align}\label{def:hR}
h_R(x) &= \int_{B(0,R)} H(x,w) [u_M(1,w)] ^{1+q}dw,\\\label{def:tildehR}
\tilde h_R(x) &= \int_{B(0,R)} \tilde H(x,w) [u_M(1,w)] ^{1+q}dw,\\\label{def:tildeHR}
\tilde H_R(x) &= \int_{B(0,R)^c} \tilde H(x,w) u_M(1,w)dw.
\end{align}
Note that $H(x,w)\le\tilde H(x,w)$ and $h_R(x)\le\tilde h_R(x)$.

\begin{lemma}\label{lem:intHp}
For $x\in\RR^d$, we have
\begin{equation}
\int_{\RR^d}  H(x,w) p(1,w) dw = C_2\, p(1,x),
\label{eq:intHp}
\end{equation}
where
$$
C_2  =\frac{\pi}{\alpha\sin\left(\pi/\alpha\right)} > 1.
$$
\end{lemma}
\begin{proof}
By scaling property and Chapman-Kolmogorov equation for   the function $p(t,x,y)$, we get 
\begin{align}\nonumber
 \int_{\RR^d} p(1-s^\alpha,x,sw)  p(1,w)\,dz&= \int_{\RR^d} s^{-d}p(s^{-\alpha}-1,s^{-1}x,w)  p(1,w)\,dz\\\label{eq:C-K}
&=  s^{-d}p(s^{-\alpha},s^{-1}x,) =  p(1,x). 
\end{align}
Consequently,
\begin{align*}
\int_{\RR^d} H(x,w) p(1,w) dw  &= \int_0^1\int_{\RR^d}(1-s^\alpha)^{-1/\alpha} p(1-s^\alpha,x,sw) p(1,w) dw\,ds \\
 &=  p(1,x) \int_0^1(1-s^\alpha)^{-1/\alpha}  ds
=\frac1\alpha \Gamma\left(1-\frac{1}{\alpha}\right)\Gamma\left(\frac{1}{\alpha}\right) \,p(1,x)\\
&=\frac{\pi}{\alpha\sin\left(\pi/\alpha\right)}\,p(1,x),
\end{align*}
where the last equality results from the Euler's reflection formula.
\end{proof}
The next step is to  provide a Chapman-Kolmogorov-like inequality involving functions $H(x,w)$ and $\tilde H(x,w)$. At first, we present a technical lemma.
\begin{lemma}\label{lem:tech}Let $\beta>0$ be fixed. For $v\in(0,1)$, we have
$$\int_v^1 r^{-\beta}(1-r^\alpha)^{-1/\alpha}(r^\alpha-v^\alpha)^{-1/\alpha}    dr\approx v^{-\beta}(1-v)^{1-2/\alpha}.$$
\end{lemma}
\begin{proof} Denote the above integral by $I(v)$. Since $a^\gamma-b^\gamma\approx (a-b)a^{\gamma-1}$ for $a>b>0$  and $\gamma>0$ (cf. Lemma 4 in \cite{MS2012}), we have $1-r^\alpha\approx 1-r$ and $r^\alpha-v^\alpha\approx(r-v)r^{\alpha-1}$. Hence,
$$I(v) \approx \int_v^1 r^{1/\alpha-1-\beta}(1-r)^{-1/\alpha}(r-v)^{-1/\alpha}dr. $$
For $v\ge1/4$, we estimate $ r^{1/\alpha-1-\beta} \approx 1$ and substitute $r=1-u(1-v)$, which gives us
$$I(v) \approx (1-v)^{1-2/\alpha} \int_0^1 u^{-1/\alpha}(1-u)^{-1/\alpha}du\approx (1-v)^{1-2/\alpha}.$$
In the case $v<1/4$, we split the integral into $\int_v^{1/2}+\int_{1/2}^1$ and obtain
\begin{align*}
I(v) &\approx \int_v^{1/2} r^{1/\alpha-1-\beta}(r-v)^{-1/\alpha}dr+ \int_{1/2}^1 (1-r)^{-1/\alpha}dr\\
&= v^{-\beta}\int_1^{1/(2v)} u^{1/\alpha-1-\beta}(u-1)^{-1/\alpha}du+\frac{\alpha2^{(\alpha-1)/\alpha}}{\alpha-1}\\ 
&\approx v^{-\beta}+1\approx v^{-\beta}, 
\end{align*}
which is equivalent to the required formula.
\end{proof}

Since $\alpha>1$, we immediately obtain the following 
\begin{cor}
\label{lem:estint}Let $\beta>0$ be fixed. For $v\in(0,1)$, we have
$$\int_v^1 r^{-\beta}(1-r^\alpha)^{-1/\alpha}(r^\alpha-v^\alpha)^{-1/\alpha}    dr\ls v^{-\beta}(1-v)^{-1/\alpha}.$$
\end{cor}
This allows us to prove the following lemma.
\begin{lemma}\label{lem:intHH}There exists a constant $C_3>0$ such that for $x,z\in\RR^d$, we have
\begin{equation}
\int_{\RR^d}  H(x,w) \tilde H(w,z) dw \le C_3 \tilde H(x,z).
\label{eq:intHH}
\end{equation}
\end{lemma}
\begin{proof}
Scaling property and Chapman-Kolmogorov equation for the function $p(t,x,y)$ give us
\begin{align*}
&\int_{\RR^d}  H(x,w) \tilde H(w,z) dw\\  
&= \int_0^1\int_0^1\int_{\RR^d}r^{-\beta} (1-s^\alpha)^{-1/\alpha} p(1-s^\alpha,x,sw) (1-r^\alpha)^{-1/\alpha} p(1-r^\alpha,w,rz) dw\,ds\, dr\\
&= \int_0^1\int_0^1\int_{\RR^d} r^{-\beta}(1-s^\alpha)^{-1/\alpha} s^{-d} p(s^{-\alpha}-1,s^{-1}x,w) (1-r^\alpha)^{-1/\alpha} p(1-r^\alpha,w,rz) dw\,ds\, dr\\
&= \int_0^1\int_0^1 r^{-\beta}(1-s^\alpha)^{-1/\alpha}(1-r^\alpha)^{-1/\alpha} s^{-d} p(s^{-\alpha}-r^\alpha,s^{-1}x, r z)  ds \,dr\\
&= \int_0^1\int_0^1 r^{-\beta}(1-s^\alpha)^{-1/\alpha}(1-r^\alpha)^{-1/\alpha}  p(1-s^{\alpha}r^\alpha,x, sr z)  ds\, dr.
\end{align*}
Substituting $s=v/r$ in the inner integral and then using Fubini-Tonelli theorem, we get
\begin{align*}
&\int_{\RR^d}  H(x,w) \tilde H(w,z) dw\\  
&= \int_0^1p(1-v^{\alpha},x, v z)\int_v^1 r^{-\beta}(1-r^\alpha)^{-1/\alpha}(r^\alpha-v^\alpha)^{-1/\alpha}    dr\, dv.
\end{align*}
By Corollary \ref{lem:estint}, we  obtain (\ref{eq:intHH}).
\end{proof}

\begin{rem} \rm
Lemma \ref{lem:tech}, which plays an important role in the above-given proof of Lemma \ref{lem:intHH},   does not hold for $\beta=0$. This explains partly the form of the functions  $\tilde H(x,w)$ and $\tilde h_R(x)$.
\end{rem}
As a consequence, we get
\begin{cor}\label{cor:intH}
For $x\in\RR^d$, we have
\begin{align}
\int_{\RR^d} H(x,w) \tilde  h_R(w) dw &\le C_3 \tilde h_R(x) \label{eq:intHhR}, \\
\int_{\RR^d}  H(x,w) \tilde H_R(w) dw &\le C_3 \tilde H_R(x) \label{eq:intHHR}. 
\end{align}
\end{cor}

Now, we pass to the proof of the main result of the paper. 
\begin{proof}[{\bf Proof of Theorem \ref{mainthm}}]
Let 
$C_0 = C_1(C_2 \vee C_3)$.
By Lemma \ref{lem:limit u}, we may choose $\eta\in(0,1)$ and $R>0$  such that $|u_M(1,x)|<\(\frac{\eta}{C_0}\)^{1/q}$ for $|x|>R$. 
Thus,  by Lemma \ref{lem:DuhamelEst}, we have
\begin{align}\label{eq:uMdecomp}
u_M(1,x) &\le Mp(1,x) +C_1 \tilde h_R(x) + \frac{C_1\eta}{C_0} \int_{B(0,R)^c}  H(x,w) u_M(1,w) dw\\
\label{eq:uMdecompH}
&\le Mp(1,x) +C_1 \tilde h_R(x) + \frac{C_1\eta}{C_0}\, \tilde H_R(x).
\end{align}  

We put  (\ref{eq:uMdecompH}) to (\ref{eq:uMdecomp}) and, by Lemma \ref{lem:intHp} and Corollary \ref{cor:intH}, we have
\begin{align}
	u_M(1,x) &\le M p(1,x) + C_1 \tilde h_R(x) \\
	& \ \ \ + \frac{C_1\eta}{C_0}\int_{B(0,R)^c}  H(x,w)\left[Mp(1,w) + C_1\tilde h_R(w) + \frac{C_1\eta}{C_0} \tilde H_R(w)\right] dw \notag\\
	&\le M(1+\eta)p(1,x) + C_1(1+\eta) \tilde h_R(x) +  \eta^2 \tilde H_R(x)\,. \label{eq:1step}
\end{align}
Now, we put  (\ref{eq:1step})   into (\ref{eq:uMdecomp}) and, by Lemma \ref{lem:intHp} and Corrolary \ref{cor:intH}, we obtain
\begin{align*}
	u_M(1,x) &\le M(1+\eta+\eta^2)p(1,x) + C_1(1+\eta+\eta^2) \tilde h_R(x) +  \eta^3 \tilde H_R(x)\,. 
\end{align*}
Hence, by induction, 
\begin{align*}
	u_M(1,x) &\le  M \sum_{k=0}^n \eta^k p(1,x) + C_1\sum_{k=0}^n \eta^k  \tilde h_R(x) +  \eta^{n+1} \tilde H_R(x)\,. 
\end{align*}
Taking $n \to \infty$, we get
\begin{align}\label{eq:iteration}
u_M(1,x) \le \frac{M}{1-\eta} p(1,x) + \frac{C_1}{1-\eta} \tilde h_R(x).
\end{align}

Furthermore,  since   both functions $p(1,\cdot)$ and $u_M(1,\cdot)$ are continuous and nonnegative (see \cite{BKW2001}, proof of Theorem 2.1), they are comparable on every compact set. Hence, we focus only  on large values of  $|x|$.  

 We first prove the upper estimate. Let $|x|>2R$. For $|w|<R$ and $s\in(0,1)$, we have $|x-sw|>|x|/2$, and consequently $p(s,x,sw)\ls |x|^{-d-\alpha}\approx p(1,x)$. Hence,
\begin{align*}
\tilde h_R(x)&=\int _{B(0,R)} \int_0^1 r^{-\beta} (1-r^\alpha)^{-1/\alpha}p(1-r^\alpha,x,r w)[u_M(1,w)]^{1+q} dr\,dw\\
&\ls  p(1,x)\left\|u_M(1,\cdot)\right\|^{1+q}_\infty\int _{B(0,R)} \int_0^1r^{-\beta} (1-r^\alpha)^{-1/\alpha} dr\,dw\\
&\approx\left\|u_M(1,\cdot)\right\|^{1+q}_\infty\,p(1,x),
\end{align*}
which, by  (\ref{eq:iteration}), gives us 
\begin{equation}\label{est:upper}
u_M(1,x)\ls p(1,x).
\end{equation}
To provide lower estimates we show that the integral in (\ref{eq:Duhamelscale}) is much smaller than $Mp(t,x)$ for sufficiently large $|x|$.  By (\ref{est:upper}) and (\ref{est:gradp}), there are constants $C_4, C_5>0$ such that
\begin{align*}
 &\left|\alpha\int_0^1 \int_{\RR^d} b \cdot \nabla_w p(1-s^\alpha,x,sw)  [u_M(1,w)]^{q+1}\,dw\,ds\right|\\
&\hspace{20mm}\le C_4 \int_0^1 \int_{\RR^d} \left| \nabla_w p(1-s^\alpha,x,sw)  [p(1,w)]^{q+1}\right|\,dw\,ds\\
&\hspace{20mm} \le C_5 \int_0^1 \int_{B(0,\tilde R)} \frac{1}{(1-s^\alpha)^{1/\alpha}+|x-sw|} p(1-s^\alpha,x,sw)  p(1,w)\,dw\,ds\\
&\hspace{24mm} + C_5 \int_0^1 \int_{B(0,\tilde R)^c} (1-s^{\alpha})^{-1/\alpha} p(1-s^\alpha,x,sw)  [p(1,w)]^{q+1}\,dw\,ds
\end{align*}
for  any $\tilde R>0$. We take $\tilde R$ such that $p(1,w)<\(4C_5\int_0^1(1-s^\alpha)^{-1/\alpha}ds/M\)^{- 1/q}$ for $|w|>\tilde R$.   Then, for  $|x|>(2\tilde R)\vee (8C_5/M)$, we obtain
\begin{align*}
 &\left|\alpha\int_0^1 \int_{\RR^d} b \cdot \nabla_w p(1-s^\alpha,x,sw)  [u_M(1,w)]^{q+1}\,dw\,ds\right|\\
&\hspace{20mm} \le \frac{2C_5}{|x|} \int_0^1 \int_{B(0,\tilde R)}  p(1-s^\alpha,x,sw)  p(1,w)\,dw\,ds\\
&\hspace{24mm}+ \frac{M}{4\int_0^1(1-s^\alpha)^{-1/\alpha}ds} \int_0^1  (1-s^{\alpha})^{-1/\alpha} \int_{B(0,\tilde R)^c} p(1-s^\alpha,x,sw)  p(1,w)\,dw\,ds.
\end{align*}
Hence, by (\ref{eq:C-K}), we get
\begin{align*}
 \left|\alpha\int_0^1 \int_{\RR^d} b \cdot \nabla_w p(1-s^\alpha,x,sw)  [u_M(1,w)]^{q+1}\,dw\,ds\right|\le \frac{M}{2}p(1,x),
\end{align*}
which, together with (\ref{eq:Duhamelscale}), implies
$$u_M(1,x) \ge \frac M2 p(t,x).$$
The proof is complete.
\end{proof}

\begin{rem}
Using the result of  Theorem \ref{mainthm} as well as  the  method of proving it  one can show that there exists a constant $C=C(d,\alpha,M,b)$ such that
$$\left|\nabla_xu_M (t,x)\right|\leq Ct^{-1/\alpha}p(t,x).$$
However, this estimate does not seem to be optimal.
\end{rem}

\subsection*{Acknowledgements}
{The paper is partially supported by grant MNiSW IP2012 018472}

\bibliographystyle{abbrv} 
\bibliography{burgers}

\end{document}